\documentclass[12pt,reqno,sumlimits]{amsart}

\usepackage{amssymb,amscd,amsmath,epsfig}
\newtheorem{theorem}{Theorem}[section]
\newtheorem{corollary}[theorem]{Corollary}
\newtheorem{lemma}[theorem]{Lemma}
\newtheorem{proposition}[theorem]{Proposition}

\theoremstyle{definition}

\newtheorem{remark}{Remark}[section]

\renewcommand{\to}{\longrightarrow}

\newsavebox{\savepar}

\numberwithin{equation}{section}

\newcounter{labelflag} 
\newcommand{\labelon}{\setcounter{labelflag}{1}}
\newcommand{\Label}[1]{
                       \ifnum\thelabelflag=1
                          \ifmmode
                             \makebox[0in][l]{\qquad\fbox{\rm#1}}
                          \else
                             \marginpar{\vspace{0.7\baselineskip}
                                        \hspace{-1.1\textwidth}
                                        \fbox{\rm#1}}
                          \fi
                       \fi
                       \label{#1}
                      }
\labelon

\usepackage[top=3cm, bottom=1.8cm, left=2cm, right=2cm,headsep=30pt]{geometry} 
\usepackage{tikz-cd}

 \newcommand{\dir}{\partial\kern-.570em /}
 \newcommand{\dire}{\partial\kern-.570em /{}^{\rm eq}}

\DeclareMathOperator{\sech}{sech}

\begin{document}

\title{Computing the Maslov index from singularities of a matrix Riccati equation}
\author[T. McCauley]{Thomas McCauley}
\address{Department of Mathematics and Statistics\\
  Boston University}
\email{tmccaule@math.bu.edu}

\maketitle

\begin{abstract}
We study the Maslov index as a tool to analyze stability of steady state solutions to a reaction-diffusion equation in one spatial dimension.  We show that the path of unstable subspaces associated to this equation is governed by a matrix Riccati equation whose solution $S$ develops singularities when changes in the Maslov index occur.  Our main result proves that at these singularities the change in Maslov index equals the number of eigenvalues of $S$ that increase to $+\infty$ minus the number of eigenvalues that decrease to $-\infty$.
\end{abstract}

\section{Introduction}

In this paper we study the Maslov index of a path of Lagrangian subspaces as a tool to analyze the stability of a steady state solution to a reaction-diffusion equation.  In particular we are interested in using a matrix Riccati equation to describe this path.  Typically the solution $S$ to this matrix Riccati equation will develop singularities at a finite collection of times and each singularity indicates a contribution to the Maslov index.  Our main results are Theorem \ref{SpecialCaseMainTheorem} and Theorem \ref{MainTheorem}, which state that the change in Maslov index at a singularity is the number of eigenvalues of $S$ that increase to $+\infty$ minus the number of eigenvalues that decrease to $-\infty$.

Connections between the Morse index of a self-adjoint operator (the number of its negative eigenvalues) and the Maslov index have been known for some time.  For example, Smale \cite{Smale} studied elliptic operators on smooth manifolds with boundary and Bott \cite{Bott} used intersection theory to study the Morse index of geodesics.  In \cite{Arnold1} and \cite{Arnold2}, Arnol'd related classical Sturm-Liouville theory to the Maslov index and symplectic geometry, similar to the method used by Conley in \cite{Conley}.   The Maslov index was first used to study stability of steady state solutions of evolution equations in \cite{Jones}, where Jones considered a Schr\"{o}dinger-type equation in one spatial dimension.  Since then it has been used to study stability in a variety of evolution equations.  See \cite{CDB} for numerous examples and an extensive collection of references.  Most examples consider equations with one spatial dimension, where there are many tools to study stability, most notably the Evans function.  It is not known how to extend the Evans function to multiple spatial dimensions except in some special cases, for example \cite{OS}, but many are optimistic that Maslov index techniques can be used in this harder setting.  See \cite{CJM} and \cite{DJ} for recent results along these lines.

In this paper we restrict our attention to the case of one spatial dimension, with the hope that techniques from this easier case will guide research in the harder case of multiple spatial dimensions, and we consider reaction-diffusion equations of the form
\begin{align*}
u_t &= \mathcal{D}u_{xx} + f(u, x),
\end{align*}
where $x \in \mathbb{R}, t  \in \mathbb{R}^+, u(x, t) \in \mathbb{R}^n$, $\mathcal{D}$ is an invertible diagonal matrix, and $f(u, x)$ is an analytic function of $u$ and $x$ such that $D_u f$ is a symmetric matrix.  Suppose $\hat{u}$ is a steady state solution such that, as $x \to \pm \infty$, $\hat{u}$ approaches some finite limit exponentially fast and $Df_u(\hat{u}, x)$ approaches some constant coefficient matrix.  To analyze stability of $\hat{u}$ we linearize this equation about $\hat{u}$ and we consider the eigenvalue problem
\begin{align*}
\mathcal{L} u \overset{\rm def}{=} \mathcal{D}u_{xx} + Df_u(\hat{u}, x) = \lambda u,
\end{align*}
where we consider $\mathcal{L}$ an operator on some appropriate function space, like $L^2(\mathbb{R})$.  Associated to this linear second order equation is a path of Lagrangian subspaces of $\mathbb{R}^{2n}$, the path of unstable subspaces, and the Maslov index of this path can be used to detect eigenvalues of $\mathcal{L}$.

In \cite{BM}, Beck and Malham study a matrix Riccati equation that governs this path of Lagrangian subspaces.  The solution $S$ typically develops singularities at a finite collection of times, and Beck and Malham prove that at a singularity the unsigned change in Maslov index is the number of singular eigenvalues of $S$ \cite[Theorem 4.1]{BM}.  In this paper we prove that this change in Maslov index is precisely the number of eigenvalues that increase to $+\infty$ minus the number of eigenvalues that decrease to $-\infty$.

This paper is organized as follows.  In Section \ref{Background} we pose our question about the stability of a steady state solution to our reaction-diffusion equation, motivating our interest in the Maslov index of the path of unstable subspaces.  Section \ref{DynamicalSystemBackground} presents the class of reaction-diffusion equations we consider and the path of unstable subspaces related to the eigenvalue problem.  These subspaces are Lagrangian, so the unstable subspaces form a path in the Lagrangian Grassmannian manifold $\Lambda(n)$.  Section \ref{SymplecticGeoBackground} recalls some pertinent facts about this manifold.  Section \ref{MaslovIndexBackground} defines the Maslov index, a fixed-endpoint homotopy invariant of paths in $\Lambda(n)$ that counts the signed intersections of a path $W(x)$ with a fixed reference plane $V$.  The Maslov index is computed by summing the signature of the crossing form $\Gamma(W, V, x)$ over all intersections $x$.

In Section \ref{ComputeMaslovChange} we show that the path of unstable subspaces $W(x)$ is governed by a matrix Riccati equation and we investigate how singularities of the solution $S(x)$ contribute to the Maslov index.  Section \ref{RiccatiEqn} presents the matrix Riccati equation and we express the change in Maslov index in terms of the singular eigenvalues of $S$, which we denote $\mu_1, \ldots, \mu_k$.  In the case that there is one singular eigenvalue $\mu$, this expression allows us to prove\\

\noindent {\bf Theorem \ref{SpecialCaseMainTheorem}.} \emph{If $\dim W(x_0) \cap V = 1$, then the signature of $\Gamma(W, V, x_0)$ is ${\rm sign}\, \dot{\mu}$.}\\

Section \ref{HigherDimIntersections} considers the case that our path of Lagrangian subspaces intersects our reference plane in a $k$-dimensional subspace, $k > 1$, and we have\\

\noindent {\bf Theorem \ref{MainTheorem}.} \emph{The signature of $\Gamma(W, V, x_0)$ is $\# \{ \mu_j :  \dot{\mu_j} \to + \infty \} - \# \{ \mu_j : \dot{\mu_j} \to -\infty \}$.}\\

Remark \ref{OneSidedLimits} relates the behavior of $\dot{\mu}_j$ to the singular behavior of $\mu_j$. These theorems allow us to compute the change in Maslov index at singularities of $S$.

In Section 4 we study stability in two reaction-diffusion equations using the Maslov index.  For each equation we outline how Theorem \ref{SpecialCaseMainTheorem} and Theorem \ref{MainTheorem} allow us to compute the Maslov index from the singularities of the solution to a matrix Riccati equation, and the Maslov index is computed for some specific values of our parameters.

We would like to thank Margaret Beck for many helpful conversations and for suggesting this problem, which arose from her discussions with Simon Malham and Robert Marangell.

\section{The Maslov index and stability}\label{Background}

In this Section we relate the stability of a steady state solution $\hat{u}$ to a reaction-diffusion equation to the Maslov index of a path of Lagranian subspaces.  Section 2.1 poses the eigenvalue problem associated with the linearization of our reaction-diffusion equation about $\hat{u}$.  The remainder of the Section presents the Maslov index, a topological invariant that detects solutions to this eigenvalue problem.  Section \ref{SymplecticGeoBackground} recalls some basic facts about the manifold of Lagrangian subspaces of $\mathbb{R}^{2n}$, denoted $\Lambda(n)$, and Section 2.3 defines the Maslov index, an integer that counts the signed intersections of a path of Lagrangian subspaces with a fixed reference plane.

\subsection{Dynamics of a steady state solution}\label{DynamicalSystemBackground}

Consider the reaction-diffusion equation
\begin{align}\label{ReactionDiffusion}
u_t &= \mathcal{D}u_{xx} + f(u, x),
\end{align}
where $x \in \mathbb{R}, t \in \mathbb{R}^+$, $u(x, t) \in \mathbb{R}^n$, $\mathcal{D}$ is an invertible diagonal matrix, and $f(u, x)$ is analytic in $u$ and $x$.  Suppose further that $Df_u(\hat{u}, x)$ is symmetric.  We are interested in studying the stability of stationary solutions to (\ref{ReactionDiffusion}).  Suppose $\hat{u}$ is a stationary solution, i.e. a solution satisfying $\hat{u}_t \equiv 0$.  We consider the linearization of (\ref{ReactionDiffusion}) about $\hat{u}$ and the corresponding eigenvalue problem,
\begin{align}\label{LinearizedReactionDiffusion}
\mathcal{L}u \overset{\rm def}{=} \mathcal{D} u_{xx} + Df_u(\hat{u}, x) u = \lambda u.
\end{align}
Assume that as $x \to \pm \infty$, $\hat{u}$ approaches some finite limit exponentially fast and $Df_u(\hat{u}, x)$ approaches some constant matrix.  $\mathcal{L}$ is self-adjoint, considered as an operator on an appropriate function space, e.g. $L^2(\mathbb{R})$, so $\sigma(\mathcal{L}) \subset \mathbb{R}$.  We assume the essential spectrum of $\mathcal{L}$ satisfies $\sigma_{\rm ess}(\mathcal{L}) \subset \{ \lambda : \lambda < -\delta < 0 \}$ for some fixed $\delta > 0$.  Thus if $\lambda \in \sigma(\mathcal{L})$ is positive, it must be an eigenvalue.  A necessary condition for $\hat{u}$ to be stable is that $\mathcal{L}$ has no positive eigenvalues.  Hence, to find the eigenvalues of $\mathcal{L}$ we look for solutions to (\ref{LinearizedReactionDiffusion}) that lie in our appropriate function space.

Let $p = u$ and $q = u_x$.  Then
\begin{align}\label{FirstOrderSystem}
\left( \begin{array}{c}
p_x\\
q_x \end{array} \right) &= 
\left( \begin{array}{cc}
0 & I\\
\mathcal{D}^{-1}(\lambda - Df_u(\hat{u}, x)) & 0 \end{array} \right)
\left( \begin{array}{c}
p\\
q \end{array} \right),
\end{align}
so we study this first order system.  Here $I$ is the $n \times n$ identity matrix.  A solution $v = (p^T, q^T)^T$ to (\ref{FirstOrderSystem}) that decays quickly as $x \to \pm \infty$ will lie in our function space, and thus solve our eigenvalue problem.

\begin{remark}\label{AnalyticRemark}
Because $\hat{u}$ solves a differential equation with analytic coefficients, $\hat{u}$ is an analytic function.  It follows that if $(p^T, q^T)^T$ solves equation (\ref{FirstOrderSystem}) then $(p^T, q^T)^T$ is analytic as well.
\end{remark}

Let $\Phi(x, s, \lambda)$ be the fundamental solution matrix of (\ref{FirstOrderSystem}).  Following \cite{CDB}, for $\lambda \notin \sigma_{\rm ess}(\mathcal{L})$ we say the \emph{stable subspace} at $(x_0, \lambda)$ is
\begin{align*}
E^s(x_0, \lambda) \overset{\rm def}{=} \{ u \in \mathbb{R}^{2n} : \lim_{x \to \infty} \Phi(x, x_0, \lambda) u = 0 \}
\end{align*}
and the \emph{unstable subspace} at $(x_0, \lambda)$ is
\begin{align*}
E^u(x_0, \lambda) \overset{\rm def}{=} \{ u \in \mathbb{R}^{2n} : \lim_{x \to -\infty} \Phi(x, x_0, \lambda) u = 0 \}.
\end{align*}
$\lambda_0$ is an eigenvalue of $\mathcal{L}$ if and only if, for some $x_0 \in \mathbb{R}$, $E^s(x_0, \lambda_0)$ and $E^u(x_0, \lambda_0)$ have a non-trivial intersection.  To find such $\lambda_0$, we compute the Maslov index of the path $x \mapsto E^u(x, \lambda)$, for it is conjectured that the Maslov index of $E^u(x, \lambda)$ is (minus) the number of eigenvalues $\lambda_0$ such that $\lambda_0 > \lambda$.  See \cite[Section 1]{JLM} for an example of such a result for matrix Hill's equations.  Focusing on intersections of $E^u(x, \lambda)$ and $E^s(\infty, \lambda)$ offers computational advantages, as the path $E^u(x, \lambda)$ is governed by the differential equation (\ref{FirstOrderSystem}).  For this reason, our goal is to find such intersections.

In this paper we say a subspace $V \subset \mathbb{R}^{2n}$ is \emph{Lagrangian} if $\dim V = n$ and $\langle v, Jw \rangle = 0$ whenever $v,w \in V$, where
\begin{align*}
J &= \left( \begin{array}{cc}
0 & -I\\
I & 0 \end{array} \right).
\end{align*}

\begin{remark}
The bilinear form $(v, w) \mapsto \langle v, Jw \rangle$ is an example of a symplectic form on $\mathbb{R}^{2n}$.  Given any symplectic form on $\mathbb{R}^{2n}$, one may define a Lagrangian subspace.  See \cite{MS} for details.
\end{remark}

When $\mathcal{D} = I$, $E^u(x, \lambda), E^s(x, \lambda) \subset \mathbb{R}^{2n}$ are Lagrangian if both subspaces are $n$-dimensional \cite[Section 4]{CDB}. Therefore, for a fixed $\lambda$, $x \mapsto E^u(x, \lambda)$ represents a path in $\Lambda(n)$, the manifold of Lagrangian subspaces of $\mathbb{R}^{2n}$.  The Maslov index of a path in $\Lambda(n)$ is an integer that represents the total (signed) intersections with a reference plane.

When $\mathcal{D} \neq I$, $E^u(x, \lambda)$ and $E^s(x, \lambda)$ are not Lagrangian with respect to the standard symplectic form.  However, these subspaces are Lagrangian with respect to the symplectic form $\omega(v, w) = \langle v, JM w \rangle$, where $M$ is given by
\begin{align*}
M &= \left( \begin{array}{cc}
\sqrt{\mathcal{D}} & 0\\
0 & \sqrt{ \mathcal{D}} \end{array} \right).
\end{align*}
Therefore the subspaces define a path in the Lagrangian Grassmannian manifold (with respect to $\omega$) and we suspect that the results of Section 3 still hold in this case.  So that we may work with the standard symplectic form on $\mathbb{R}^{2n}$, we assume henceforth that $\mathcal{D} = I$.  In the following sections we recall important facts about $\Lambda(n)$ and we define the Maslov index.

\subsection{The Lagrangian Grassmannian manifold}\label{SymplecticGeoBackground}

This Section presents a few important facts about the Lagrangian Grassmannian manifold $\Lambda(n)$, following \cite[Chapter 2]{MS}.  Given a Lagrangian subspace $V \subset \mathbb{R}^{2n}$, we can represent $V$ by the $2n \times n$ matrix $M = ( v_1, \ldots, v_n)$, where the column vectors $\{ v_1, \ldots, v_n \}$ form a basis of $V$.  It follows that $M$ has rank $n$ and ${\rm range}\, M = V$.  Moreover, we may choose $\{ v_1, \ldots, v_n \}$ to be an orthonormal basis.  We define the $n \times n$ matrices $X$ and $Y$ by
\begin{align}\label{SubspaceRepresentation}
\left( \begin{array}{cccc}
v_1 & v_2 & \ldots & v_n \end{array} \right) &=
\left( \begin{array}{c}
X\\
Y \end{array} \right).
\end{align}
Let $Z = X + iY$.  In \cite[Section 2.3]{MS}, McDuff and Salamon show that $Z \in U(n)$.  Conversely, given $Z  =  X + iY \in U(n)$, the range of the matrix $(X^T, Y^T)^T : \mathbb{R}^n \to \mathbb{R}^{2n}$ is a Lagrangian subspace of $\mathbb{R}^{2n}$.  In this way, $Z$ represents a Lagrangian subspace.  Moreover, for any $O \in O(n)$, the unitary matrix $ZO$ represents the same Lagrangian subspace, and if $Z$ and $Z'$ represent the same subspace then they differ by some $O \in O(n)$.  It follows that $\Lambda(n)$ is diffeomorphic to the homogeneous space $U(n)/O(n)$.  With this diffeomorphism, we may (non-uniquely) represent a Lagrangian subspace $V$ as a unitary matrix $Z = X + i Y$.

Furthermore, \cite[Lemma 2.30]{MS} proves that the $2n \times n$ matrix
\begin{align*}
\left( \begin{array}{c}
I\\
A \end{array} \right)
\end{align*}
represents a Lagrangian subspace if and only if $A$ is symmetric.  It follows that, in equation (\ref{SubspaceRepresentation}), if $X$ is invertible, $YX^{-1}$ is symmetric.

\subsection{The Maslov index}\label{MaslovIndexBackground}

In this Section we define the Maslov index for a path in $\Lambda(n)$, following the exposition in \cite[Chapter 2]{MS} and \cite{RS}.  We are primarily concerned with the Maslov index of the path of unstable subspaces $E^u(x, \lambda) : (-\infty, \infty) \to \Lambda(n)$.  Since $Df_u(\hat{u}, x)$ approaches some constant matrix as $x \to \pm \infty$, the limits $\lim_{x \to \infty }E^u(x, \lambda)$ and $\lim_{x \to -\infty}E^u(x, \lambda)$ exist.  Therefore we may extend $E^u(x, \lambda)$ to be a path whose domain is the extended real line $\mathbb{R} \cup  \{ \pm \infty \}$.  Because $\mathbb{R} \cup \{ \pm \infty \}$ is homeomorphic to a compact interval, we may define the Maslov index for paths whose domain is $[0, 1]$.

Suppose $Z(x) = X(x) + i Y(x) : [0, 1] \to U(n)$ is a path of unitary matrices representing a path $W(x)$ of Lagrangian subspaces in $\mathbb{R}^{2n}$.  Suppose further $W(x)$ has a nontrivial intersection with a fixed Lagrangian subspace $V$ at $x_0$.  The \emph{crossing form} at $x_0$ is the quadratric form given by
\begin{align}\label{CrossingForm}
\Gamma(W, V, x_0) u &= \langle X(x_0)u, \dot{Y}(x_0)u \rangle - \langle Y(x_0)u, \dot{X}(x_0)u \rangle,
\end{align}
restricted to the subspace of $u \in \mathbb{R}^n$ satisfying $(X(x_0)^T, Y(x_0)^T)^T u \in W(x_0) \cap V$.  Because $(X(x_0)^T, Y(x_0)^T)^T$ is rank $n$, there is a unique $u$ for each $v \in W(x_0) \cap V$.  Henceforth we let $V$ be the vertical subspace ${\rm span}\{ e_{n + 1}, \ldots, e_{2n} \}$.  $W(x_0)$ has non-trivial intersection with $V$ precisely when $\det X(x_0) = 0$ and $\dim \ker X(x_0) = \dim W(x_0) \cap V$ \cite[Section 2.3]{MS}.  In this case, $(X(x_0)^T, Y(x_0)^T)^Tu \in W(x_0) \cap V$ when $u \in \ker X(x_0)$, so that the crossing form becomes
\begin{align*}
\Gamma(W, V, x_0)u = - \langle Y(x_0) u, \dot{X}(x_0) u \rangle.
\end{align*}
We remark that we may also represent $W(x)$ by the path $( I,  S(x))^T$, where $S = YX^{-1}$, whenever $X$ is invertible.  For this reason, $S(x)$ develops a singularity when $W(x)$ has nontrivial intersection with $V$.

The \emph{signature} of a non-degenerate quadratic form $Q$ is the number of positive eigenvalues minus the number of negative eigenvalues of the associated symmetric matrix.  Note that we do not distinguish between the quadratic form and its associated symmetric matrix.  We say that $x_0$ is a \emph{regular crossing} if $\Gamma(W, V, x_0)$ is a non-degenerate quadratic form.  Suppose $W(x)$ has only regular crossings, which must occur at a finite collection of times $x_1 < x_2 < \ldots < x_N$ since regular crossings are isolated \cite[\textsection 2]{RS}.  Following \cite{RS}, we define the \emph{Maslov index} of the path $W(x)$ to be
\begin{align*}
{\rm Mas}(W(x);V) = \sum_{j = 1}^N {\rm sgn}\, \Gamma( W, V, x_j).
\end{align*}

Let $\Sigma_k(V) \subset \Lambda(n)$ be the collection of Lagrangian subspaces $V'$ such that $\dim V \cap V' = k$.  The \emph{Maslov cycle} is $\Sigma(V) = \bigcup_{k = 1}^n \Sigma_k(V)$, an algebraic variety in $\Lambda(n)$.  In the literature $\Sigma(V)$ is also called the \emph{train} of $V$.  The highest stratum, $\Sigma_1(V)$, has codimension 1 in $\Lambda(n)$.

\begin{remark}\label{GenericIntersection}
A path in $\Lambda(n)$ generically intersects $\Sigma(V)$ in the highest stratum, $\Sigma_1(V)$ \cite[\textsection 2.3]{MS}.  It follows that a non-trivial intersection of $W(x)$ and $V$ is generically 1-dimensional.
\end{remark}

Suppose we wish to compute the Maslov index of $W(x)$ with respect to a reference plane $V'$ different from the vertical subspace.  Typically it suffices to find ${\rm Mas}(W(x); V)$ because of

\begin{proposition}
Let $W(x) : [0, 1] \to \Lambda(n)$ be a path and let $W_0 = W(0)$ and $W_1 = W(1)$.  Suppose that $W_0, W_1 \in \Sigma_0(V) \cap \Sigma_0(V')$.  Then ${\rm Mas}(W(x) ; V) = {\rm Mas}(W(x) ; V') + s( V, V' ; W_0, W_1 )$, where $s(V, V' ; W_0, W_1)$ is the Maslov index of the loop consisting of a curve in $\Sigma_0(V)$ from $W_1$ to $W_0$ followed by a curve in $\Sigma_0(V')$ from $W_0$ to $W_1$.
\end{proposition}

The integer $s(V, V' ; W_0, W_1)$ is called the \emph{H\"{o}rmander index}.  This Proposition appears in \cite[Equation (2.9)]{Duistermaat} without proof, so we provide one now.  This proof will rely on the following two properties of the Maslov index.  First, suppose $\gamma : [0, 1] \to \Lambda(n)$ is a loop.  Then ${\rm Mas}(\gamma; V) = {\rm Mas}(\gamma; V')$.  This follows from the fact that the Maslov index for loops has an equivalent definition that does not depend on the choice of a reference plane \cite[Section 2.3]{MS}.

Second, suppose $\gamma_1, \gamma_2 : [0, 1] \to \Lambda(n)$ are two curves with $\gamma_1(1) = \gamma_2(0)$.  The concatenated curve $\gamma_2 \cdot \gamma_1$ satisfies
\begin{align*}
{\rm Mas}(\gamma_2 \cdot \gamma_1 ; V') = {\rm Mas}(\gamma_1; V') + {\rm Mas}(\gamma_2; V').
\end{align*}
This is proven in \cite[Theorem 2.3]{RS}.

\begin{proof}
Let $\gamma_1$ be the curve in $\Sigma_0(V)$ from $W_1$ to $W_0$ and let $\gamma_2$ be the curve in $\Sigma_0(V')$ from $W_0$ to $W_1$.  Because $\gamma_1$ lies in $\Sigma_0(V)$, ${\rm Mas}(\gamma_1; V) = 0$ and 
\begin{align*}
{\rm Mas}(W(x) ; V) &= {\rm Mas}(W(x) ; V) + {\rm Mas}( \gamma_1 ; V) = {\rm Mas}(\gamma_1 \cdot W(x) ; V).
\end{align*}
Furthermore, $\gamma_1 \cdot W(x)$ is a loop, so ${\rm Mas}(\gamma_1 \cdot W(x) ; V) = {\rm Mas}(\gamma_1 \cdot W(x) ; V')$.  Since $\gamma_2$ lies in $\Sigma_0(V')$, ${\rm Mas}(\gamma_2; V') = 0$ and we have
\begin{align*}
{\rm Mas}(\gamma_1 \cdot W(x) ; V') &= {\rm Mas}(\gamma_1 \cdot W(x); V') + {\rm Mas}(\gamma_2; V')\\
&= {\rm Mas}(W(x); V') + {\rm Mas}(\gamma_2 \cdot \gamma_1; V').
\end{align*}
By definition, $s(V, V' ; W_0, W_1) = {\rm Mas}(\gamma_2 \cdot \gamma_1 ; V')$, so we have proven our claim.
\end{proof}

\section{Singularities of a matrix Riccati equation and the crossing form}\label{ComputeMaslovChange}

In this section we show that the path $W(x)$ determined by equation (\ref{FirstOrderSystem}) is governed by a matrix Riccati equation whose solution $S$ may develop singularities at a finite collection of times.  Theorems \ref{SpecialCaseMainTheorem} and \ref{MainTheorem} state that at singularities the contribution to the Maslov index is the number of singular eigenvalues that increase to $+\infty$ minus the number of singular eigenvalues that decrease to $-\infty$.  Section 3.1 presents this matrix Riccati equation and relates its solution to the signature of the crossing form.  We immediately have Theorem \ref{SpecialCaseMainTheorem}, which computes the change in Maslov index at intersections where $\dim W(x) \cap V = 1$, which is the generic case.  Section 3.2 generalizes this result to higher dimensional intersections in Theorem \ref{MainTheorem}.

\subsection{Matrix Riccati equation}\label{RiccatiEqn}

We consider a path of Lagrangian subspaces $W(x)$ determined by a differential equation
\begin{align}\label{MatrixODE}
 \left( \begin{array}{c}
\dot{X}\\
\dot{Y} \end{array} \right) &=
\left( \begin{array}{cc}
A &  B \\
C & D \end{array} \right)
\left( \begin{array}{c}
X\\
Y \end{array} \right).
\end{align}
We are primarily concerned with equation (\ref{FirstOrderSystem}) but we consider this more general equation with an eye towards applying this method to a larger class of evolution equations discussed in Section \ref{ConcludingRemarks}.  Equation (\ref{MatrixODE}) implies $\dot{X} = AX + BY$ and $\dot{Y} = CX + DY$.  Letting $S = YX^{-1}$ we have
\begin{align}\label{FirstRiccatiEqn}
\dot{S} &= \dot{Y}X^{-1} + Y \dot{(X^{-1})} = (CX + DY)X^{-1} - YX^{-1}\dot{X}X^{-1} \nonumber \\
&= C + D YX^{-1} - YX^{-1}(AX + BY) X^{-1}\\
&= C + D YX^{-1} - YX^{-1} A + YX^{-1}BYX^{-1} \nonumber \\
&= C + DS - SA + SBS, \nonumber
\end{align}
showing that $S$ satisfies a matrix Riccati equation.

By equation (\ref{CrossingForm}), the crossing form is determined by the symmetric matrix $X(x_0)^T\dot{Y}(x_0) - Y(x_0)^T\dot{X}(x_0)$ restricted to those $u$ such that $(X(x_0)^T, Y(x_0)^T)^Tu \in W(x_0) \cap V$, since
\begin{align*}
\langle X(x_0) u, \dot{Y}(x_0) u \rangle - \langle Y(x_0), \dot{X}(x_0) u \rangle &= \langle u, \big(X(x_0)^T\dot{Y}(x_0) - Y(x_0)^T \dot{X}(x_0)\big) u \rangle.
\end{align*}
Suppose $X(x)$ is invertible.  Then, letting $S = YX^{-1}$, 
\begin{align*}
(X^{-1})^T(X^T\dot{Y} - Y^T\dot{X})X^{-1} &= \dot{Y}X^{-1} - (X^{-1})^TY^T \dot{X}X^{-1}\\
= \dot{Y}X^{-1} - (YX^{-1})^T\dot{X} X^{-1} &= \dot{Y}X^{-1}  - YX^{-1} \dot{X} X^{-1}\\
= \dot{Y}X^{-1}  + Y \dot{(X^{-1})} &= \dot{S}
\end{align*}

The following theorem is a classical result about change of coordinates for quadratic forms, whose proof can be found in \cite[Section 2.5]{Gerstein}.

\begin{theorem}[Sylvester's Law of Inertia] \label{SylvestersLaw}
Let $A$ be a real $n \times n$ symmetric matrix and let $U$ be a real invertible matrix.  Then ${\rm sgn}\, A = {\rm sgn}\, U^T A U$.
\end{theorem}

Suppose $\dim W(x_0) \cap V = k$.  Then $\dim \ker X(x_0) = k$ and $S(x)$ has $k$ singular eigenvalues $\mu_1(x), \ldots, \mu_k(x)$ \cite[Theorem 4.1]{BM}.  As discussed in Section \ref{SymplecticGeoBackground}, $S(x)$ is symmetric for each $x$ so we may take an orthonormal frame $\{ w_j(x) \}_{j = 1}^n$ for $\mathbb{R}^n$ such that $S(x)w_j(x) = \mu_j(x)w_j(x)$.  For any $k$-dimensional subspace $\widehat{W} \subset \mathbb{R}^n$,  Theorem \ref{SylvestersLaw} implies that when $X$ is invertible,
\begin{align}\label{CoordChange}
{\rm sgn}\, (X^T \dot{Y} - Y^T \dot{X})\big|_{\widehat{W}} &= {\rm sgn}\, (X^{-1})^T(X^T \dot{Y} - Y^T \dot{X})X^{-1}\big|_{X(\widehat{W})} = {\rm sgn}\, \dot{S}\big|_{X(\widehat{W})}.
\end{align}

If $\{ w_j(x) \}_{j = 1}^k$ are eigenvectors corresponding to the singular eigenvalues $\{ \mu_j(x) \}_{j = 1}^k$ of $S(x)$, we let ${\rm span}\{w_1(x), \ldots, w_k(x) \} = \overline{W}(x) \subset \mathbb{R}^n$.

\begin{lemma}\label{AnalyticEigenvectors}
$w_j(x)$ may be chosen to be analytic.
\end{lemma}

\begin{proof}
By Remark \ref{AnalyticRemark}, $X$ and $Y$ are analytic.  Therefore $S = YX^{-1}$ is analytic with a singularity at $x_0$.  Cramer's rule says that $(X^{-1})_{ij} = (\det X)^{-1} (-1)^{i + j} M_{ji}$, where $M_{ji}$ is the determinant of $X$ with the $j^{th}$ row and $i^{th}$ column removed.  Thus $M_{ji}$ is analytic,  implying that $(\det X) X^{-1}$ is analytic near $x_0$, having no singularity, as $( (\det X)X^{-1})_{ij} = (-1)^{i + j} M_{ji}$.  Therefore $(\det X)S$ is analytic as well.

In particular, $(\det X)S$ is analytic and symmetric, so by \cite[II \textsection 6.2]{Kato} the eigenvectors of $(\det X)S$ may be chosen to be analytic.  Since $S$ and $(\det X)S$ have the same eigenvectors, we have shown that we may choose $w_j(x)$ to be analytic.
\end{proof}

\begin{corollary}\label{ContinuousSubspaces}
$\overline{W}(x)$ is continuous.
\end{corollary}

\begin{lemma}\label{LimitOfKernel}
$X(x)^{-1}\overline{W}(x) \to \ker X(x_0)$ as $x \to x_0$.
\end{lemma}

\begin{proof}
Since $Sw_j = \mu_j w_j$ and $| \mu_j | \to \infty$ as $x \to x_0$, $\| S w_j \| = \| Y X^{-1} w_j \| \to \infty$ as $x \to x_0$.  Note that $Y$ is continuous near $x_0$, so $\| Y \|$ is continuous near $x_0$ as well.  Moreover, $\| YX^{-1} w_j \| \leq \| Y \| \| X^{-1} w_j \|$, implying $\| X^{-1} w_j \| \to \infty$ as $x \to x_0$. Therefore
\begin{align*}
X\left( \frac{X^{-1}w_j}{\| X^{-1} w_j \|} \right) &= \frac{w_j}{\| X^{-1}w_j \| } \to 0
\end{align*}
as $x \to x_0$, which means that ${\rm span}\{ X^{-1} w_j \} \to \ker X(x_0)$ as $x \to x_0$.  Since $\overline{W}(x) = {\rm span} \{ w_1(x), \ldots, w_k(x) \}$, we have shown that $X(x)^{-1}\overline{W}(x) \to \ker X(x_0)$ as $x \to x_0$.
\end{proof}

\begin{lemma}\label{SignatureLimit}
Let $x_0 \in (a, b)$, $A(x) : (a, b) \to M_{n \times n}(\mathbb{R})$ be a continuous path of symmetric matrices, and $W(x) : (a, b) \to {\rm Gr}(k, n)$ be a path of $k$-dimensional subspaces in the Grassmannian manifold of $k$-planes in $\mathbb{R}^n$.  Suppose $A(x_0)$ is invertible.  Then ${\rm sgn}\, A(x)|_{W(x)} \to {\rm sgn}\, A(x_0)|_{W(x_0)}$ as $x \to x_0$.
\end{lemma}

\begin{proof}
Near $x_0$ we may pick a basis $\{ v_j(x) \}$ of $W(x)$ such that each $v_j(x)$ is continuous.  We may represent the quadratic form determined by $A(x)|_{W(x)}$ by the matrix $\alpha$ whose $(i, j)$-entry is $\alpha_{ij}(x) \overset{\rm def}{=} \langle v_i(x), A(x)v_j(x) \rangle$.  Since $A(x)$ and $v_i(x), v_j(x)$ are continuous, $A_{ij}(x)$ is continuous as well.  It follows that the path of quadratic forms determined by $A(x)|_{W(x)}$ is represented by a continuous path of symmetric matrices.  Therefore ${\rm sgn} \, A(x)|_{W(x)} \to A(x_0)|_{W(x_0)}$  as $x \to x_0$, as the signature of a quadratic form is a locally constant function on the space of non-degenerate quadratic forms.
\end{proof}

\begin{proposition}\label{SignatureLimitForS}
For $x$ close to $x_0$, ${\rm sgn}\, \Gamma(W, V, x_0) = {\rm sgn}\, \dot{S}|_{\overline{W}(x)}$.
\end{proposition}

\begin{proof}
Putting together Lemmas \ref{LimitOfKernel} and \ref{SignatureLimit} with equation (\ref{CoordChange}), we have
\begin{align*}
{\rm sgn}\, \dot{S}\big|_{\overline{W}(x)} &= {\rm sgn}\, (X^T\dot{Y} - Y^T \dot{X}) \big|_{X^{-1}\overline{W}(x)} \to {\rm sgn}\, (X^T\dot{Y} - Y^T \dot{X})\big|_{\ker X},
\end{align*}
which is the signature of $\Gamma(W, V, x_0)$.
\end{proof}

Proposition \ref{SignatureLimitForS} and Corollary \ref{ContinuousSubspaces} require that $S(x)$ depends analytically on $x$, which is guaranteed if, in equation (\ref{ReactionDiffusion}), $f(u, x)$ depends on $u$ and $x$ analytically.  This condition is satisfied in all applications we have in mind; see, for example, \cite{CDB}.  If $f(u, x)$ is merely smooth in $u$ or $x$, $S(x)$ may not depend analytically on $x$, which means $\overline{W}(x)$ may not be continuous and Proposition \ref{SignatureLimitForS} may not hold.  See \cite[Chapter II, \textsection 5.3]{Kato} for details.

Differentiating the equation $S(x)w_j(x) = \mu_j(x)w_j(x)$ we find
\begin{align*}
\dot{S}(x)w_j(x) + S(x)\dot{w}_j(x) = \dot{\mu}_j(x)w_j(x) + \mu_j(x)\dot{w}_j(x).
\end{align*}
Moreover, differentiating the equation $\langle w_j(x), w_k(x) \rangle = \delta_{jk}$ we find that 
\begin{align*}
\langle \dot{w}_j(x), w_k(x) \rangle + \langle w_j(x), \dot{w}_k(x) \rangle = 0.
\end{align*}
Therefore
\begin{align*}
\langle w_j, \dot{S}w_k \rangle &= \langle w_j, \dot{\mu}_k w_k \rangle + \langle w_j, \mu_k \dot{w}_k \rangle - \langle w_j, S \dot{w}_k \rangle = \dot{\mu}_k \delta_{jk} + \mu_k \langle w_j, \dot{w}_k \rangle - \langle Sw_j, \dot{w}_k \rangle \\ 
&= \dot{\mu}_k \delta_{jk} + \mu_k \langle w_j, \dot{w}_k \rangle - \mu_j \langle w_j, \dot{w}_k \rangle = \dot{\mu}_k \delta_{jk} + (\mu_k - \mu_j) \langle w_j, \dot{w}_k \rangle.
\end{align*}
Letting $g_{jk} = \langle w_j, \dot{w}_k \rangle$ we have shown
\begin{align}\label{SummaryEqn}
\langle w_j,  \dot{S}w_k \rangle &= \delta_{jk} \dot{\mu}_k + (\mu_k - \mu_j) g_{jk}.
\end{align}

By equation (\ref{SummaryEqn}), we can represent the quadratic form determined by $\dot{S}|_{\overline{W}(x)}$ by the $k \times k$ matrix
\begin{align*}
Q &= \left( \begin{array}{ccccc}
\dot{\mu}_1 & (\mu_2 - \mu_1)g_{12} & (\mu_3 - \mu_1)g_{13} &  \ldots & (\mu_k - \mu_1)g_{1k}\\
(\mu_1 - \mu_2) g_{21} & \dot{\mu}_2 & (\mu_3 - \mu_1)g_{23} & \ldots & (\mu_k - \mu_2)g_{2k}\\
(\mu_1 - \mu_3) g_{31} & (\mu_2 - \mu_3)g_{32} & \dot{\mu}_3 & \ldots & (\mu_k - \mu_3)g_{2k}\\
\vdots & & & \ddots &  \vdots\\
(\mu_1 - \mu_k)g_{k1} & (\mu_2 - \mu_k)g_{k2} & (\mu_3 - \mu_k)g_{3k} & \ldots & \dot{\mu}_k 
\end{array} \right).
\end{align*}

Generically $W(x)$ intersects $V$ in a 1-dimensional subspace.  In this case, $S$ has one singular eigenvalue $\mu$ and $Q = (\dot{\mu})$.

\begin{theorem}\label{SpecialCaseMainTheorem}
If $\dim W(x_0) \cap V = 1$, then the signature of $\Gamma(W, V, x_0)$ is ${\rm sign}\, \dot{\mu}$.
\end{theorem}

\begin{proof}
When $\dim W(x_0) \cap V = 1$, there is a single singular eigenvalue $\mu$.  The quadratic form determined by $\dot{S}|_{W'(x)}$ is given by $Q = (\dot{\mu})$.  Therefore ${\rm sgn} \, Q = {\rm sign}\, \dot{\mu}$.
\end{proof}

The remainder of this Section considers the case that $W(x)$ intersects $V$ in a $k$-dimensional subspace, $k >1$.

\subsection{Higher dimensional intersections}\label{HigherDimIntersections} 

The proof of Theorem \ref{SpecialCaseMainTheorem} is simple because $Q$ has no off-diagonal entries when $\dim W(x_0) \cap V = 1$.  For higher dimensional intersections, the off-diagonal terms $(\mu_j - \mu_k)g_{kj}$ may play an important role.  Nevertheless, the $g_{kj}$ factor does not introduce an additional singularity, because of

\begin{lemma}
$g_{jk}(x)$ are analytic near $x_0$.
\end{lemma}

\begin{proof}
Since each $w_j$ is analytic, $\dot{w}_j$ is analytic as well and so $g_{jk} = \langle w_j, \dot{w}_k \rangle$ is analytic.
\end{proof}

Since $S(x)$ is symmetric for each $x$, there is some path $Q(x)$ of orthogonal matrices such that $Q(x)^TS(x)Q(x) = \tilde{S}(x)$ is diagonal.  Using equation (\ref{FirstRiccatiEqn}) we have
\begin{align*}
\dot{\tilde{S}} &= \dot{Q}^TSQ + Q^T\dot{S}Q + Q^TS\dot{Q}\\
&= \dot{Q}^TQ Q^T SQ + Q^T(C + DS - SA + SBS) Q + Q^T S Q Q^T \dot{Q}\\
&= \dot{Q}^TQ Q^TSQ + Q^TCQ + Q^TDQQ^TSQ - Q^TSQ Q^TAQ + \\
&\indent+(Q^TSQ)Q^TBQ (Q^TSQ) + Q^TSQ Q^T \dot{Q}.
\end{align*}
Letting 
\begin{align*}
\tilde{A} &= Q^TAQ, \hphantom{bbbb} \tilde{B} = Q^TBQ, \hphantom{bbbb} \tilde{C} = Q^TCQ,\\
\tilde{D} &= Q^TDQ, \hphantom{bbbb} M = \dot{Q}^TQ = -Q^T\dot{Q},
\end{align*}
we may rewrite this equation as
\begin{align*}
\dot{\tilde{S}} &= M \tilde{S} - \tilde{S}M + \tilde{C} + \tilde{D}\tilde{S} - \tilde{S}\tilde{A} + \tilde{S}\tilde{B}\tilde{S}.
\end{align*}
$\tilde{S}$ is a diagonal matrix whose entries are the eigenvalues of $S$, $\{ \mu_j \}$.  Looking at the diagonal entries in this matrix equation we see that the eigenvalues satisfy the scalar Riccati equations
\begin{align}\label{GeneralCaseScalarRiccati}
\dot{\mu}_{j} &= \tilde{c}_{jj} + \tilde{d}_{jj}\mu_{j} - \tilde{a}_{jj}\mu_{j} + \tilde{b}_{jj}\mu_{j}^2,
\end{align}
for $j  = 1, \ldots, n$.  Notice that the diagonal entries of the term $M\tilde{S} - \tilde{S}M$ all vanish because $\tilde{S}$ is diagonal and $M$ is skew-symmetric.  This differential equation governs the order of singularity that $\mu_j$ admits.

\begin{proposition}\label{GeneralTamedSingularity}
Suppose $\mu_j$ has a singularity at $x_0$ and $\tilde{b}_{jj} \neq 0$ near $x_0$.  Then the singularity has order one.
\end{proposition}

\begin{proof}
We first note that $A, B, C, D$ are assumed analytic and therefore they are bounded near $x_0$.  Since $Q$ is an orthogonal matrix, $|Q_{jk}| \leq 1$ for each entry $Q_{jk}$.  Therefore $\tilde{A}, \tilde{B}, \tilde{C}, \tilde{D}$ are bounded near $x_0$.

In the proof of Lemma \ref{AnalyticEigenvectors}, we show that $(\det X)S$ is analytic and without singularity at $x_0$.  Thus its eigenvalues $\{ (\det X)\mu_j \}$ are analytic and without singularity.  Therefore we can write $\mu_j = f_j/g_j$ with $f_j$ and $g_j$ analytic near $x_0$, by setting $f_j = (\det X)\mu_j$ and $g_j = \det X$.  In particular, we see that $\mu_j$ has a pole at $x_0$ of finite order.

Suppose $\mu_j$ has a singularity of order $k$.  A straightforward calculation shows that $\dot{\mu}_j$ has a singularity of order $k + 1$.  Then the right hand side of equation (\ref{GeneralCaseScalarRiccati}) has a singularity of order $2k$, because $\tilde{b}_{jj} \neq 0$ near $x_0$.  Because both sides of the equation have singularities of the same order, $k + 1 = 2k$, implying $k = 1$.
\end{proof}

Recall that we are primarily interested in equation (\ref{FirstOrderSystem}) in which $A = D = 0$ and $B = I$.  In this case $\tilde{B} = I$ and $\tilde{b}_{jj} = 1$ for all $j$.  Therefore the hypotheses of Proposition \ref{GeneralTamedSingularity} are satisfied and we have

\begin{theorem}\label{MainTheorem}
The signature of $\Gamma(W, V, x_0)$ is $\# \{ \mu_j : \dot{\mu_j} \to + \infty \} - \# \{\mu_j : \dot{\mu_j} \to -\infty \}$.
\end{theorem}

\begin{proof}
Because $g_{jk}$ is analytic near $x_0$ and $\mu_j - \mu_k$ blows up with at most order 1, we see that $(x - x_0)(\mu_k - \mu_j)g_{jk}$ approaches a finite number as $x \to x_0$.  It follows that $(x - x_0)^2(\mu_k - \mu_j)g_{jk} \to 0$ as $x \to x_0$.  Similarly, $\dot{\mu}_j$ has a singularity of order 2 at $x_0$, implying $(x - x_0)^2 \dot{\mu}_j \to m_j $ as $x \to x_0$, for some non-zero $m_j \in \mathbb{R}$.  We may write
\begin{align*}
Q(x) &= \left( \begin{array}{ccccc}
\dot{\mu}_1 & (\mu_2 - \mu_1)g_{12} & (\mu_3 - \mu_1)g_{13} &  \ldots & (\mu_k - \mu_1)g_{1k}\\
(\mu_1 - \mu_2) g_{21} & \dot{\mu}_2 & (\mu_3 - \mu_1)g_{23} & \ldots & (\mu_k - \mu_2)g_{2k}\\
(\mu_1 - \mu_3) g_{31} & (\mu_2 - \mu_3)g_{32} & \dot{\mu}_3 & \ldots & (\mu_k - \mu_3)g_{2k}\\
\vdots & & & \ddots &  \vdots\\
(\mu_1 - \mu_k)g_{k1} & (\mu_2 - \mu_k)g_{k2} & (\mu_3 - \mu_k)g_{3k} & \ldots & \dot{\mu}_k 
\end{array} \right)\\
&= \frac{1}{(x - x_0)^2}\left( \begin{array}{cccc}
(x - x_0)^2\dot{\mu}_1 & (x - x_0)^2(\mu_2 - \mu_1)g_{12} &  \ldots & (x - x_0)^2(\mu_k - \mu_1)g_{1k}\\
(x - x_0)^2(\mu_1 - \mu_2) g_{21} & (x - x_0)^2\dot{\mu}_2 &  \ldots & (x - x_0)^2(\mu_k - \mu_2)g_{2k}\\
(x - x_0)^2(\mu_1 - \mu_3) g_{31} & (x - x_0)^2(\mu_2 - \mu_3)g_{32}  & \ldots & (x - x_0)^2(\mu_k - \mu_3)g_{2k}\\
\vdots & & \ddots &  \vdots\\
(x - x_0)^2(\mu_1 - \mu_k)g_{k1} & (x - x_0)^2(\mu_2 - \mu_k)g_{k2} & \ldots & (x - x_0)^2\dot{\mu}_k 
\end{array} \right)\\
&= \frac{1}{(x - x_0)^2} \bar{Q}(x)
\end{align*}
Notice that ${\rm sgn}\, Q(x) = {\rm sgn}\, \bar{Q}(x)$, since the eigenvalues of ${Q}(x)$ and $\bar{Q}(x)$ have the same signs.  As $x \to x_0$, the discussion above shows that
\begin{align*}
\bar{Q}(x) \longrightarrow \left( \begin{array}{ccccc}
m_1 & 0 & 0 & \ldots & 0\\
0 & m_2 & 0 &  & 0\\
0 & 0 & m_3 &  & 0\\
\vdots & & & \ddots & \vdots\\
0 & 0 & 0 & \ldots & m_k \end{array} \right) = {\rm diag}(m_1, \ldots, m_k)
\end{align*}
as $x \to x_0$.  Since the signature of a nondegenerate symmetric matrix is locally constant, ${\rm sgn}\, \bar{Q}(x) = {\rm sgn}\,{\rm diag}(m_1, \ldots, m_k)$ for $x$ close to $x_0$.  Moreover, because $m_j$ and $\dot{\mu}_j$ have the same sign,
\begin{align*}
{\rm sgn}\, {\rm diag}(m_1, \ldots, m_k) = \# \{\mu_j : \dot{\mu}_j \to + \infty \} - \# \{\mu_j : \dot{\mu}_j \to -\infty \}.
\end{align*}
Therefore, if $x$ is sufficiently close to $x_0$,
\begin{align*}
{\rm sgn} Q(x) = {\rm sgn}\bar{Q}(x) &= \# \{\mu_j : \dot{\mu}_j \to + \infty \} - \# \{ \mu_j : \dot{\mu}_j \to -\infty \}.
\end{align*}
Because ${\rm sgn}\, \Gamma(W, V, x_0) = {\rm sgn}\, Q(x)$ for $x$ close to $x_0$, we have proven our claim.

\end{proof}

\begin{remark}\label{OneSidedLimits}
Because the eigenvalues $\mu_j$ have singularities of order one at $x_0$, $\dot{\mu}_j \to \infty$ if and only if $\lim_{x \to x_0^-}\mu_j = \infty$ and $\lim_{x \to x_0^+}\mu_j = -\infty$, and similarly $\dot{\mu}_j \to -\infty$ if and only if $\lim_{x \to x_0^-}\mu_j = -\infty$ and $\lim_{x \to x_0^+}\mu_j = \infty$.  In practice it may be easier to determine the contribution to the Maslov index at $x_0$ from these one-sided limits, as they can be determined from a plot of $\mu_j$ near $x_0$.  We employ this strategy in Section \ref{Examples}.
\end{remark}

\section{Examples}\label{Examples}

In this section we present two examples of reaction-diffusion equations with steady state solutions whose stability we study by Maslov index techniques and the matrix Riccati equation.  Both examples come from \cite{CDB}.

\subsection{A 1-dimensional example}

This example comes from \cite[Section 5]{CDB}.  We consider the reaction-diffusion equation
\begin{align}\label{FirstPDE}
u_t &= u_{xx} - u + u^2,
\end{align}
for the scalar valued function $u(x, t)$.  The function $\hat{u}(x) = \frac{3}{2} \sech^2(\frac{x}{2})$ is a steady state solution.  Thus we may linearize (\ref{FirstPDE}) about $\hat{u}$ to obtain the eigenvalue equation
\begin{align}\label{FirstLinearizedPDE}
\mathcal{L}u &\overset{\rm def}{=} u_{xx} - u + 2\hat{u}(x) u = \lambda u.
\end{align}
Chardard, Dias, and Bridges use the Evan's function to show the eigenvalues of $\mathcal{L}$ are $\{ -3/4, 0, 5/4 \}$.  We now demonstrate how the Maslov index can be used to detect this positive eigenvalue.  Setting $p = u$ and $q = u_x$ we may convert equation (\ref{FirstLinearizedPDE}) into the first order system
\begin{align*}
\left( \begin{array}{c}
p_x\\
q_x \end{array} \right) &=
\left( \begin{array}{cc}
0 & 1\\
1 + \lambda - 2\hat{u}(x) & 0 \end{array} \right)
\left( \begin{array}{c}
p\\
q \end{array} \right).
\end{align*}
Using the notation $\frac{d}{ds}h = h_s$, we see the stable and unstable subspaces are given by
\begin{align*}
E^s(x, \lambda) &= \left( \begin{array}{c}
2h^-\\
h_s^{-} - \gamma h^- \end{array} \right), \hphantom{bbb}  E^u(x, \lambda) = \left( \begin{array}{c}
2h^+\\
h_s^+ + \gamma h^+ \end{array} \right),
\end{align*}
where we identify the $2 \times 1$ matrices with their images in $\mathbb{R}^2$.  Here $s = x/2$, $\gamma = 2\sqrt{ \lambda + 1}$, and
\begin{align*}
h^{\pm}(s) = \pm a_0 + a_1 \tanh(s) \pm a_2 \tanh^2(s) + \tanh^2(s),
\end{align*}
where
\begin{align*}
a_0 &= \frac{\gamma}{15} (4 - \gamma^2), \hphantom{qqqq} a_1 = \frac{1}{5}(2\gamma^2 - 3), \hphantom{qqqq} a_2 = -\gamma.
\end{align*}
Thus
\begin{align*}
E^u(x, \lambda) &= \left( \begin{array}{c}
2h^+\\
h_s^+ + \gamma h^+ \end{array} \right) = \left( \begin{array}{c}
X\\
Y \end{array} \right)
\end{align*}
and $S = Y/X = (h_s^+ + \gamma h^+)/2h^+$.  Since $S$ is a $1 \times 1$ matrix we can identify it with its eigenvalue $\mu$, and singularities of $\mu$ indicate a change in the Maslov index.  The following figure plots $\mu$ when $\lambda = 1$.\\

\begin{center}
\includegraphics[height = 1.5in, width = 2in]{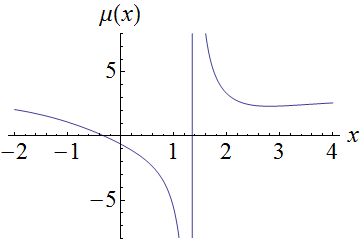}
\end{center}

In this case, $\mu$ has a singularity near $x = 1.34981$, and the plot indicates a change in Maslov index by $-1$.  Since $\mu$ has no other singularities, we conclude that ${\rm Mas}(E^u(x, 1); V) = -1$.

A straightforward computation shows
\begin{align*}
E^u(\pm \infty, \lambda) &= \lim_{x \to \pm \infty} E^u(x, \lambda) = \left( \begin{array}{c}
1\\
\sqrt{ \lambda + 1} \end{array} \right),\\
E^s(\infty, \lambda) &= \lim_{x \to \infty} E^s(x, \lambda) = \left( \begin{array}{c}
1\\
-\sqrt{ \lambda + 1} \end{array} \right)
\end{align*}
so that $s( E^s(\infty, \lambda), V; E^u(-\infty, \lambda), E^u(\infty, \lambda) ) = 0$.  Therefore 
\begin{align*}
{\rm Mas}(E^u(x, \lambda); V) = {\rm Mas}(E^u(x, \lambda); E^s(\infty, \lambda)),
\end{align*}
and we conclude that ${\rm Mas}(E^u(x, 1); E^s(\infty, 1)) = -1$.  Consistent with the conjecture discussed in Section \ref{DynamicalSystemBackground}, the Maslov index detects the single eigenvalue larger than $\lambda = 1$ and we conclude that $\hat{u}$ is unstable.

\subsection{A 2-dimensional example}

This example comes from \cite[Section 11]{CDB}.  We consider the system
\begin{align}\label{NonDecoupledSystem}
u_t &= u_{xx} - 4u + 6u^2 - c(u - v)\\
v_t &= v_{xx} - 4v + 6v^2 + c(u - v) \nonumber
\end{align}
where $c > -2$.  The functions $u = v = \hat{u} \overset{\rm def}{=} \sech^2(x)$ form a steady solution to this system.  Linearizing about $\hat{u}$ leads us to consider the eigenvalue problem
\begin{align}
u_{xx}  + 12 \sech^2(x) u &= (\lambda + 4 + c )u - cv\\
v_{xx} + 12\sech^2(x) v &= -cu + (\lambda + 4 + c)v \nonumber
\end{align}
Moreover, we may perform a change of variables $u = \tilde{u} - \tilde{v}$, $v = \tilde{u} + \tilde{v}$ to transform (\ref{NonDecoupledSystem}) into the decoupled system
\begin{align}
\tilde{u}_{xx} + 12 \sech^2(x) \tilde{u} &= (\lambda + 4) \tilde{u},\\
\tilde{v}_{xx} + 12 \sech^2(x) \tilde{v} &= (\lambda + 4 + 2c) \tilde{v}. \nonumber
\end{align}
In \cite[Appendix B]{CDB}, Chardard, Bridges, and Dias state that this system has the solutions
\begin{align*}
\tilde{u}^\pm  &= e^{\pm \sqrt{ \lambda + 4}\, x}( \pm a_0 + a_1 \tanh(x) \pm a_2 \tanh^2(x) + \tanh^3(x) )\\
\tilde{v}^\pm  &= e^{\pm \sqrt{ \lambda + 4 + 2c}\, x}( \pm b_0 + b_1 \tanh(x) \pm b_2 \tanh^2(x) + \tanh^3(x) )
\end{align*}
where
\begin{align*}
a_0 &= - \frac{ \lambda \sqrt{ \lambda + 4}}{15}, \hphantom{qqqq} a_1 = \frac{1}{5}(2\lambda + 5), \hphantom{qqqq} a_2 = -\sqrt{\lambda + 4},\\
b_0 &= - \frac{ (\lambda + 2c)\sqrt{ \lambda + 4 + 2c}}{15}, \hphantom{qqq} b_1 = \frac{1}{5}(2\lambda + 4c + 5), \hphantom{qqq} b_2 = - \sqrt{ \lambda + 4 + 2c},
\end{align*}
provided $\lambda + 4 > 0$ and $\lambda + 4 + 2c > 0$.  Thus we have
\begin{align*}
E^u(x, \lambda) &= \left( \begin{array}{cc}
\tilde{u}^+ & 0\\
0 & \tilde{v}^+\\
\tilde{u}_x^+ & 0\\
0 & \tilde{v}_x^+ \end{array} \right) 
= \left( \begin{array}{c}
X\\
Y \end{array} \right),\\
S &= YX^{-1} = \left( \begin{array}{cc}
\tilde{u}_x^+/\tilde{u}^+ & 0\\
0 & \tilde{v}_x^+/\tilde{v}^+ \end{array} \right).
\end{align*}
The eigenvalues of $S$ are $\mu_1 = \tilde{u}_x^+/\tilde{u}^+$ and $\mu_2 = \tilde{v}_x^+/\tilde{v}^+$, plotted below for $\lambda = 1$ and $c = -1$.
\begin{center}
\begin{tabular}{ccc}
\includegraphics[height = 1.5in, width = 2in]{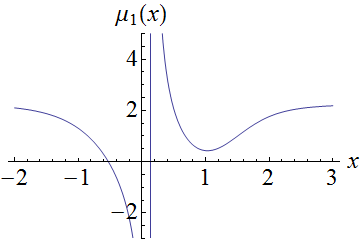} & \hphantom{bbb}&
\includegraphics[height = 1.5in, width = 2in]{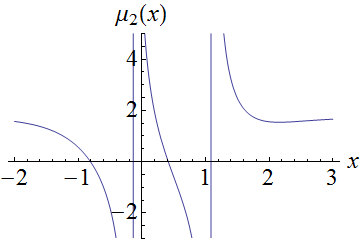}
\end{tabular}
\end{center}

These plots show that $\mu_1$ has one singularity contributing $-1$ to the Maslov index and $\mu_2$ has two singularities, each contributing $-1$ to the Maslov index.  Therefore ${\rm Mas}(E^u(x, 1); V) = -3$.

A straightforward computation shows
\begin{align*}
E^u(\pm \infty, \lambda) = \lim_{x \to \pm \infty} E^u(x, \lambda) = 
\left( \begin{array}{cc}
1 & 0\\
0 & 1\\
\sqrt{\lambda + 4} & 0\\
0 & \sqrt{\lambda + 4 + 2c} \end{array} \right),\\
E^s(\infty, \lambda) = \lim_{x \to \infty} E^s(x, \lambda) =
\left( \begin{array}{cc}
1 & 0\\
0 & 1\\
-\sqrt{\lambda + 4} & 0\\
0 & -\sqrt{ \lambda + 4 + 2c} \end{array} \right),
\end{align*}
implying 
\begin{align*}
s(E^s(\infty, \lambda), V ; E^u(-\infty, \lambda), E^u(\infty, \lambda) ) = 0,
\end{align*}
and therefore ${\rm Mas}(E^u(x, 1) ; E^s(\infty, 1)) = {\rm Mas}(E^u(x, 1); V ) = -3$.  The Maslov index detects that when $c = -1$ there are three eigenvalues greater $\lambda = 1$, which agrees with computations in \cite[Section 11]{CDB} showing the positive eigenvalues are $\lambda = 2, 3, 5$.  We conclude that $(\hat{u}, \hat{u})$ is an unstable steady state solution.

\section{Concluding Remarks}\label{ConcludingRemarks}

In this paper we have focused on reaction-diffusion equations of the form (\ref{ReactionDiffusion}), while Maslov index techniques have been applied to a much larger class of evolution equations.  In \cite{CDB}, for example, the Maslov index is used to study stability in the fifth order Korteweg-de Vries equation and a model PDE for long-wave-short-wave resonance.  In these cases, the process of relating an eigenvalue problem to the Maslov index of the path of unstable subspaces $E^u(x, \lambda)$ goes through unchanged;  after linearizing about a steady state solution and considering the associated eigenvalue problem, we transform our second order equation into a first order system of the form (\ref{MatrixODE}).

An important point is that when we begin with a reaction-diffusion equation, our first order system has $B = I$, implying that the hypotheses of Proposition \ref{GeneralTamedSingularity} are satisfied.  When considering this larger class of PDEs we may have $B \neq I$, so Proposition \ref{GeneralTamedSingularity} and consequently Theorem \ref{MainTheorem} may not apply.  Nevertheless, in these cases one may check a fortiori that the hypotheses of Proposition \ref{GeneralTamedSingularity} are true, and in this way Theorem \ref{MainTheorem} may be used to calculate the Maslov index.

Furthermore, Proposition \ref{GeneralTamedSingularity} is not used in the proof of Theorem \ref{SpecialCaseMainTheorem}, which describes the contribution to the Maslov index at 1-dimensional intersections of $E^u(x, \lambda)$ with the reference plane and so it applies to the larger class of evolution equations which Maslov index techniques have been applied to. By Remark \ref{GenericIntersection}, intersections are generically 1-dimensional so that Theorem \ref{SpecialCaseMainTheorem} suffices to compute the Maslov index in the majority of examples, even in this larger class of PDEs.

\end{document}